 \documentclass[preprint,12pt]{elsarticle}

\usepackage{amssymb}

\usepackage{graphicx} 

\usepackage{amsmath,amssymb,amsthm}

\usepackage{amsthm}

\usepackage{enumitem}

\theoremstyle{plain}
\newtheorem{theorem}{Theorem}
\newtheorem{corollary}[theorem]{Corollary}
\newtheorem{proposition}[theorem]{Proposition}
\newtheorem{lemma}[theorem]{Lemma}

\theoremstyle{definition}
\newtheorem{definition}[theorem]{Definition}
\newtheorem{assumption}[theorem]{Assumption}
\newtheorem{example}[theorem]{Example}

\theoremstyle{remark}
\newtheorem{remark}[theorem]{Remark}

\newcommand{\norm}[1]{\lVert #1 \rVert}

\newcommand{\nnnorm}[1]{{\left\vert\kern-0.25ex\left\vert\kern-0.25ex\left\vert #1 \right\vert\kern-0.25ex\right\vert\kern-0.25ex\right\vert}}
\newcommand{\abs}[1]{\lvert #1 \rvert}
\newcommand{\Abs}[1]{\left\lvert #1 \right\rvert}

\usepackage{soul,color}
\setstcolor{red}
\setul{0.5ex}{0.2ex}
\usepackage{proofread}
\noproofreadmark

\begin{document}

\begin{frontmatter}



\title{Stability analysis of linear systems subject to regenerative switchings}


\author[ttu]{M.~Ogura\corref{cor}}
\ead{masaki.ogura@ttu.edu}

\author[ttu]{C.~F.~Martin}
\ead{clyde.f.martin@ttu.edu}

\cortext[cor]{Corresponding author}

\address[ttu]{Department of Mathematics and Statistics,
Texas Tech University, Broadway and Boston, 
Lubbock, TX 79409-1042, USA}

\begin{abstract}
This paper investigates the stability of switched linear systems whose
switching signal is modeled as a stochastic process called a
regenerative process. We show that the mean stability of such a
switched system is characterized by the spectral radius of a matrix.
The matrix is obtained by taking the expectation of the transition
matrix of the system on one cycle of the underlying regenerative
process. The characterization generalizes Floquet's theorem for the
stability analysis of linear time-periodic systems. We illustrate the
result with the stability analysis of a linear system with a
failure-prone controller under periodic maintenance.
\end{abstract}


\end{frontmatter}


\section{Introduction}

The stability analysis of stochastic switched linear systems has
attracted a significant amount of attention in the last two decades.
In particular, \change{a lot of attentions has been paid to}{a lot of
effort has been put on} their mean stability, which requires that the
power of the norm of the state variable converges to zero in
expectation. Some early results on the mean stability of switched
linear systems with an independent and identically distributed
(i.i.d.)~switching signal can be found
in~\cite{Kalman1957,Bertram1959}. \change{Initiated by the results in
[3, 4], a great deal of efforts has been put on analyzing the mean
stability of so-called Markov jump linear systems~{\cite{Costa2013}},
which are switched linear systems whose switching signal is a
homogeneous Markov processes.}{The stability
characterizations~\cite{Feng1992,Costa1993} of linear systems subject
to switching by homogeneous Markov process now form the basis of the
various types of optimal control of so-called Markov jump linear
systems~{\cite{Costa2013}}.} The stability characterizations of
switched linear systems driven by an extension of homogeneous Markov
processes called \add{homogeneous} semi-Markov
processes~\cite{Janssen2006} are available in
\cite{Ogura2013f,Antunes2013}.

It is known that \del{both continuous-time}{\delws}homogeneous Markov
processes having certain irreducibility and recurrence properties and
\add{also} discrete-time i.i.d.~stochastic processes are special cases
of a more general class of stochastic processes called {\it
regenerative processes}~\cite{Sigman1993}. Firstly introduced by
Smith~\cite{Smith1955}, regenerative processes have found applications
especially in queuing systems~\cite{Zhang2008b} and network
reliability analysis~\cite{Logothetis1997}. As we will see later in
Example~\ref{eg:process}, regenerative processes are also suitable to
describe a controlled system under periodic
maintenance~\cite{Canfield1986,Nakagawa1986,Gertsbakh2000}. Despite
the above facts, as far as we are aware of, no effort has been made to
investigate switched linear systems with a regenerative switching
signal \add{in the literature of systems and control theory}.

The aim of this paper is to give the characterization of the mean
stability of a switched linear system with a regenerative switching
signal, which we call a {\it regenerative switched linear system}. We
show that, if the exponent of the mean stability is even or the system
is positive~\cite{Ogura2013f,Bolzern2014}, then the mean stability of
the system is characterized by the spectral radius of a matrix. The
matrix is obtained as the expected value of the
lift~\cite{Parrilo2008} of the transition matrix of the system
\add{over one cycle of the underlying regenerative process}. The proof
makes use of a stability-preserving discretization of the system at
the embedded renewal process of the underlying regenerative process.
The characterization in particular generalizes well-known
Floquet's theorem for the stability analysis of linear time-periodic
systems~\cite{Farina2000}.

This paper is organized as follows. After preparing necessary
notations and conventions, in Section~\ref{sec:regsystems} we recall
the definition of regenerative processes and then introduce
regenerative switched linear systems. Then
Section~\ref{sec:mainresult} presents the main result of this paper,
which is followed by an example. The proof of the main result is given
in Section~\ref{sec:proof}. Then Section~\ref{sec:disc} discusses the
discrete-time case.

\subsection{Mathematical preliminaries}

Let $(\Omega, \mathcal M, P)$ be a probability space. For a{n}
{integrable} random variable $X$ on $\Omega$ its expected value is
denoted by $E[X]$. The random variables that appear in this paper will
be assumed to be integrable.

When $x\in \mathbb{R}^n$ is nonnegative entrywise we write $x\geq 0$.
The standard Euclidean norm on $\mathbb{R}^n$ is denoted by
$\norm{\cdot}$. For \del{a general}{\delws}$m\geq 1$, the $m$-norm on
$\mathbb{R}^n$ is defined by $\norm{x}_m = (\sum_{i=1}^n
\abs{x_i}^m)^{1/m}$. The symbol $1_n$ denotes the column vector of
length $n$ whose entries are all $1$. It is easy to see $\norm{x}_1 =
1_n^\top x$ if $x\geq 0$. Let $I$ and $O$ denote the identity and the
zero matrices, respectively. We say that $A \in \mathbb{R}^{n\times
n}$ is Schur stable if its spectral radius $\rho(A)$ is less than one.

The $m$-lift of~$x \in \mathbb{R}^n$, denoted by~$x^{[m]}$, is defined
\cite{Parrilo2008} as the real vector of length~$n_m =
\binom{n+m-1}{m}$ with its elements being the lexicographically
ordered monomials~$\sqrt{\alpha!}\,x^\alpha$ indexed by all the
possible exponents $\alpha = (\alpha_1, \dotsc, \alpha_n) \in \{0, 1,
\dotsc, m\}^n$ such that \mbox{$\alpha_1 + \cdots + \alpha_n = m$},
where $\alpha! := {m!}/(\alpha_1! \dotsm \alpha_n!)$. It
holds~\cite{Parrilo2008} that
\begin{equation}\label{eq:norm^[m]}
\norm{x^{[m]}} = \norm{x}^m. 
\end{equation}
We then define $A^{[m]} \in \mathbb{R}^{n_m\times n_m}$ as the
unique matrix~\cite{Parrilo2008} satisfying $(Ax)^{[m]} = A^{[m]}
x^{[m]}$ for every $x\in\mathbb{R}^n$. For any matrix $B$ it holds
that
\begin{equation}\label{eq:(AB)^[m]}
(AB)^{[m]} = A^{[m]} B^{[m]}
\end{equation}
provided the product $AB$ is well defined. We also define $A_{[m]} \in
\mathbb{R}^{n_m \times n_m}$ as the unique real matrix
\cite{Brockett1973,Barkin1983} such that, for every
$\mathbb{R}^n$-valued differentiable function~$x$ {on $\mathbb{R}$}
satisfying $dx/dt = Ax$, it holds that ${dx^{[m]}}/{dt} = A_{[m]}
x^{[m]}$. It is easy to check that
\begin{equation}\label{eq:(e^At)^[m]}
\bigl(e^{At}\bigr)^{[m]} = e^{A_{[m]}t}
\end{equation}
for every $t\geq 0$.

\section{Regenerative switched linear systems} \label{sec:regsystems}
 
Let us first recall the definition of regenerative stochastic
processes~\cite{Sigman1993}. Throughout this paper we fix an
underlying probability space $(\Omega, \mathcal M, P)$.

\begin{definition}
A stochastic process $\sigma = \{\sigma_t\}_{t\geq 0}$ is called a
{\it regenerative process} if there exists a random variable~$R_1 >
0$, called a regeneration epoch, such that the following statements
hold.
\begin{itemize}
\item $\{\sigma_{t+R_1}  \}_{t\geq 0}$ is independent of $\{\{
\sigma_t \}_{t<R_1}, R_1\}$;

\item $\{\sigma_{t+R_1}  \}_{t\geq 0}$ is stochastically equivalent to
$\{\sigma_t \}_{t\geq 0}$.
\end{itemize}
\end{definition}

In the following we quote some consequences of the above definition
from \cite{Sigman1993}. By repeatedly applying the definition, one can
obtain a sequence of independent and identically distributed random
variables $\{R_k\}_{k\geq 1}$ called {\it cycle lengths}, which can be
used to break $\sigma$ into independent and identically distributed
{\it cycles} $\{\sigma_t \}_{0\leq t< R_1}$, $\{ \sigma_{t}
\}_{R_1\leq t< R_1 + R_2}$, $\dotsc$. Then the stochastic
process~$\{Z_k\}_{k\geq 1}$ defined by $Z_k = R_1 + \cdots + R_k$ is
called  the {\it embedded renewal process} of~$\sigma$. Throughout
this paper, for the sake of convenience, we set $Z_0 = 0$ and call
$\{Z_k\}_{k\geq 0}$ as the embedded renewal process of $\sigma$.

The next example presents a regenerative process that is not a
homogeneous Markov process \change{but}{and} is of a systems and
control theoretical interest.

\newcommand{\bm}[1]{\begin{bmatrix}#1\end{bmatrix}}
\begin{example}\label{eg:process}
Consider a dynamical system with a failure-prone
controller~\cite{Zhang2008a}. Let us model the controlled system as a
switched system with the two modes $\{1, 2\} = \{ \text{Non-failure},
\text{Failure} \}$. Instead of assuming that the transition of the
mode \del{of the system}{\delws}can be described by a homogeneous
Markov process (see, e.g., \cite{Bolzern2010a,Mariton1989}), let us
consider the scenario when the controlled system is under periodic
maintenance, which is commonly employed in the literature from
reliability theory~\cite{Canfield1986,Nakagawa1986,Gertsbakh2000}.

Let the stochastic process~$\{Z_k\}_{k\geq 0}$  represent the times at
which a maintenance is performed. For simplicity we assume that
$\sigma_{Z_k} = 1$ with probability one for every $k\geq 0$, i.e.,
that every maintenance repairs a failure with probability one within a
negligible time period. We set $Z_0 = 0$. Furthermore we assume that
$R_k = Z_{k}-Z_{k-1}$ equals $T + \varDelta_k$, where $T>0$ is a
constant and $\{\varDelta_k\}_{k=1}^\infty$ are independent and
identically distributed random variables. $T$ represents the designed
period of the  maintenance and $\varDelta_k$ models its random
perturbation. We assume that the length of the time for which the
process $\sigma$ stays at mode $1$ after the reset at
$t=Z_k$ follows an exponential distribution with parameter~$\lambda>0$.
In other words we are assuming that, on any interval of a sufficiently
small length $h$, the probability of the occurrence of a failure is
approximately equal to $\lambda h$. Then $\sigma$ is clearly a
regenerative process with a regeneration epoch $R_1$ and the embedded
renewal process $\{Z_k\}_{k\geq 0}$.

Notice that $\sigma$ is not a \del{homogeneous}{\delws}Markov process because
the length of the time while mode $2$ is active depends on \add{past information, i.e., }\change{the
random variables $\{Z_k\}_{k=1}^\infty$}{the length of the last time
interval while mode $1$ was active}.
\end{example}

Then we introduce the class of switched linear systems studied in this
paper. Let $\{\sigma_t \}_{t\geq 0}$ be a regenerative process that
can take values in a set~$\mathcal S$. Let $\{ A_s\}_{s\in \mathcal S}$ be
a family of real $n\times n$ matrices indexed by $\mathcal S$. Then we call the
stochastic differential equation
\begin{equation*} 
\Sigma : 
\dfrac{dx}{dt} = A_{\sigma_t} x(t)
\end{equation*}
as a {\it regenerative switched linear system}. We assume that $x(0) =
x_0 \in \mathbb{R}^n$ is a constant vector.

The stability of $\Sigma$ is defined in the following standard
manner. 
\begin{definition} 
Let $m$ be a positive integer. 
\begin{itemize}
\item $\Sigma$ is said to be \emph{exponentially $m$th mean stable}
if there exist $\alpha>0$ and $\beta > 0$ such that $E[\norm{x(t)}^m]
\leq \alpha e^{-\beta \mathchangescript{k}{t}}\norm{x_0}^m$ for all
$x_0$ \add{and $t\geq 0$}.

\item $\Sigma$ is said to be \emph{stochastically $m$th mean stable}
if $\int_0^\infty E[\norm{x(t)}^m]\,dt < \infty$  for any $x_0$.
\end{itemize}
\end{definition}

We also introduce the notion of positivity for $\Sigma$ following
\cite{Ogura2013f,Bolzern2014}.

\begin{definition}
We say that $\Sigma$ is {\it positive} if $x_0 \geq 0$ implies $x(t)
\geq 0$ with probability one for every $t \geq 0$.
\end{definition}

For $\Sigma$ to be positive it is clearly sufficient that all the
matrices $A_s$ are Metzler, i.e., the off-diagonal entries of each
$A_s$ are all nonnegative~\cite{Farina2000}. However it is not
necessary, as illustrated in the following non-trivial example.

\begin{example}
Consider a switched linear system with $\mathcal S = \{1,
2\}$ and
\begin{equation*}
A_1 = \frac{1}{2}\begin{bmatrix}
1&1\\1&1
\end{bmatrix},\ 
A_2 = \begin{bmatrix}
0&1\\-1&0
\end{bmatrix}.
\end{equation*}
Since $e^{A_1t} = I + (e^t - 1)A_1$, a simple calculation shows the
existence of $T>0$ such that if $t\geq T$ then, for every $x_0 \geq
0$, the vector $e^{A_1t}x_0$ is in the sector $S = \{x\in
\mathbb{R}^2: x\geq 0,\ \arg x \geq 1 \}$. Then we construct a
regenerative process $\sigma$ as follows. Set $R_1 = T + 1$ and let
$h$ follow the uniform distribution on $[T, T+1]$. Define $\sigma$ on
the first cycle $[0, R_1)$ by
\begin{equation*}
\sigma_t = 
\begin{cases}
1,&0\leq t\leq h\\
2, &h\leq t< R_1
\end{cases}
\end{equation*}
and extend this definition to the whole interval $[0, \infty)$
regeneratively. \add{We can see that the above defined $\sigma$ is a
regenerative process as in Example~\ref{eg:process}.} Then, since the
second mode decreases the argument of the state vector at most $R_1 -
h <1$, we can see that $x(t)$ stays in the positive orthant  for every
$t\geq 0$ whenever $x_0\geq 0$. Therefore $\Sigma$ is positive
although $A_2$ is not Metzler.
\end{example}

\section{Stability characterization} \label{sec:mainresult}

This section states the characterization of the mean stability of
regenerative switched linear systems and also presents an example to
illustrate the result. We state the next assumption.

\begin{assumption}\ 
\begin{enumerate}[label={A\arabic*. },ref={A\arabic*}]
\item \label{item:evenORnon-neg} Either $m$ is even or $\Sigma$ is
positive.

\item \label{item:hk<T} \change{There exists $T\geq 0$ such that $R_1 \leq T$
with probability one}{$R_1$ is essentially bounded}.

\item \label{item:Mbdd} The set $\{ A_s\}_{s\in \mathcal S}$ is bounded.
\end{enumerate}
\end{assumption}

\ref{item:evenORnon-neg} covers mean square stability ($m=2$), which
has been the central stability notion of stochastic switched linear
systems in the literature~\cite{Feng1992,Antunes2013,Fang2002a}. The
second condition~\ref{item:hk<T} on the boundedness of cycle lengths
is crucial. Similar assumptions were employed for the stability
analysis of semi-Markov jump linear systems~\cite{Ogura2013f} and
stochastic hybrid systems with renewal transitions~\cite{Antunes2013}.
\ref{item:Mbdd} is only to ensure that the state variable does not
diverge in a finite time and thus is not restrictive.

In order to state the main result we need fundamental
matrices~\cite{Farina2000} of the system $\Sigma$. For all $\omega \in
\Omega$, $t_0 \geq 0$, and $t\geq t_0$ let us define $\Phi(\omega;
t_0, t) \in \mathbb{R}^{n\times n}$ by the differential equation
\begin{equation*}
\frac{\partial \Phi}{\partial t}
= 
A_{\sigma_t(\omega)} \Phi(\omega; t_0, t) 
,\  
\Phi(\omega; t_0, t_0) = I_n.
\end{equation*}
Then define the $\mathbb{R}^{n\times n}$-valued random variables
$\{M_k\}_{k\geq 0}$ by
\begin{equation}\label{eq:def:M_k}
M_k(\omega) := \Phi(\omega; Z_k(\omega), Z_{k+1}(\omega)), 
\end{equation}
which expresses the transition of $x$ from $t=Z_k$ to $t=Z_{k+1}$.

The next theorem is the main result of this paper.

\begin{theorem}\label{theorem:main}
The following statements are equivalent.
\begin{enumerate}
\item \label{item:expsta} $\Sigma$ is exponentially $m$th mean
stable.

\item \label{item:stosta} $\Sigma$ is stochastically $m$th mean
stable.

\item \label{item:Schsta} $E[M_0^{[m]}]$ is Schur stable.
\end{enumerate}
\end{theorem}

Based on the theorem and continuing from Example~\ref{eg:process}, the
next example presents the stability analysis of a linear
time-invariant system with a failure-prone controller under periodic
maintenance.

\begin{example}
Consider the internally unstable linear time-invariant system $dx/dt =
Ax + Bu$ with the failure-prone controller
\begin{equation*}
u(t) = \begin{cases}
0,\quad \text{a fault is occuring}\\
Kx(t),\quad \text{otherwise}
\end{cases}
\end{equation*}
where 
\begin{equation*}
A = \begin{bmatrix}-0.4 &0.2\\ -0.1&0.5
\end{bmatrix},
\ 
B = \begin{bmatrix}
0\\1
\end{bmatrix},
\ 
K = \begin{bmatrix}
-0.1& -1.6
\end{bmatrix}.
\end{equation*}
The stabilizing feedback gain $K$ is obtained by solving a linear
quadratic regulator problem. We assume that the transition between the
modes $\{1, 2\} = \{ \text{Non-failure}, \text{Failure} \}$ follows
the regenerative process~$\sigma$ described in
Example~\ref{eg:process}. With this labeling we have $A_1 = A+BK$ and
$A_2 = A$.  For simplicity \add{we set $\lambda = 1$} and also
\del{we}{\delws}suppose that each $\Delta_k$ follows \change{a}{the} uniform
distribution on \add{the interval}~$[-0.1 T, 0.1 T]$ independently.

Let the random variable $h$ denote the first time in \change{$[0,
\infty)$}{$[0, R_1)$} when the transition to mode $2$ occurs. \add{We
set $h = R_1$ when a transition does not occur on the interval.} Then
one can see $M_0 = e^{A_{2}\max(0, R_1-h)} e^{A_{1}\min(R_1,h)}$. If
we let $\bar A_i:= (A_i)_{[m]}$ ($i=1,2$) then
equations~\eqref{eq:(e^At)^[m]} and \eqref{eq:(AB)^[m]} show
\begin{equation*}
M_0^{[m]} = e^{\bar A_2 \max(0, R_1-h)} e^{\bar A_1 \min(R_1,h)}. 
\end{equation*}
Here we recall that for square matrices $F_1, F_2$ with the same
dimensions and $t\geq 0$ it holds that~\cite{Chen1994}
\begin{equation*}
\exp\left(
\begin{bmatrix}
F_1& I\\O&F_2
\end{bmatrix}t
\right)
=
\begin{bmatrix}
* & \int_0^t e^{(t-\tau)F_1}e^{\tau F_2}\,d\tau
\\
0 & *
\end{bmatrix}.
\end{equation*}
Using this identity and the independence of $h$ and $R_1$, since $h$
follows the exponential distribution with mean $1$, we can
show that
\begin{equation*}
\begin{aligned}
E[M_0^{[m]}]
&=
\int_{0.9 T}^{1.1 T} \int_0^\infty 
e^{\bar A_2\max(0, t-s)} e^{\bar A_1\min(t,s)}
 e^{-s}ds\,\frac{dt}{0.2 T} 
\\
&=
\frac{5}{T}\int_{0.9 T}^{1.1 T} \int_0^t
e^{\bar A_2(t-s)} e^{(\bar A_1 - I)s}
\,ds\,{dt}
+
\frac{5}{T}\int_{0.9 T}^{1.1 T} \int_t^\infty
e^{\bar A_1t} e^{-s}
\,ds\,{dt}
\\
&=
\frac{5}{T} \begin{bmatrix}
I &O
\end{bmatrix}
\int_{0.9 T}^{1.1 T} 
\exp\left(\begin{bmatrix}
\bar A_2&I\\O&\bar A_1 - I
\end{bmatrix}t\right)
\,dt \begin{bmatrix}
O \\ I
\end{bmatrix}
\\&
\hspace{2cm}+
\frac{5}{T}\int_{0.9 T}^{1.1 T} 
e^{(\bar A_1 - I)t}\,{dt}.
\end{aligned} 
\end{equation*}
Figure~\ref{figure:stability} shows the graph of the spectral radius
of $E[M_0^{[2]}]$ as $T$ varies \change{from $1$ to $3$}{from $0$ to
$2$}.
\begin{figure}[tp]
\centering
\includegraphics[width=8.1cm]{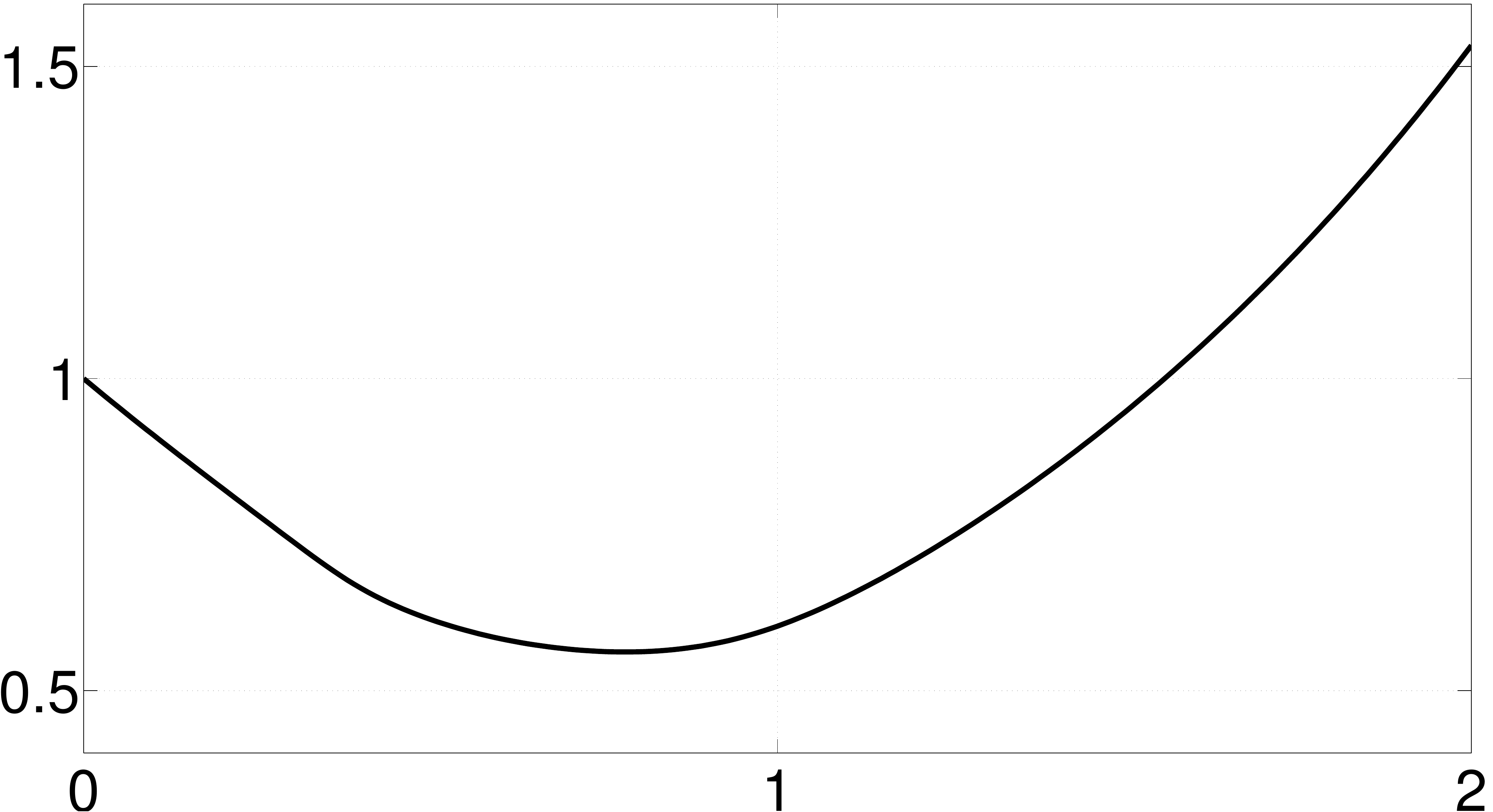}
\caption{Spectral radius of $E[\Phi^{[2]}]$ as $T$ varies}
\label{figure:stability}
\end{figure}
As is expected, instability is caused by making the period of the
maintenance longer. We can see that, by Theorem~\ref{theorem:main},
$\Sigma$ is mean square stable if and only if $T<1.55$. The
computation of the matrix $E[M_0^{[2]}]$ is performed with MATLAB.
Figure~\ref{figure:samplepaths} shows 20 sample paths of
$\norm{x(t)}^{\add 2}$ when $T=1.25$.
\begin{figure}[tp]
\centering
\includegraphics[width=8.1cm]{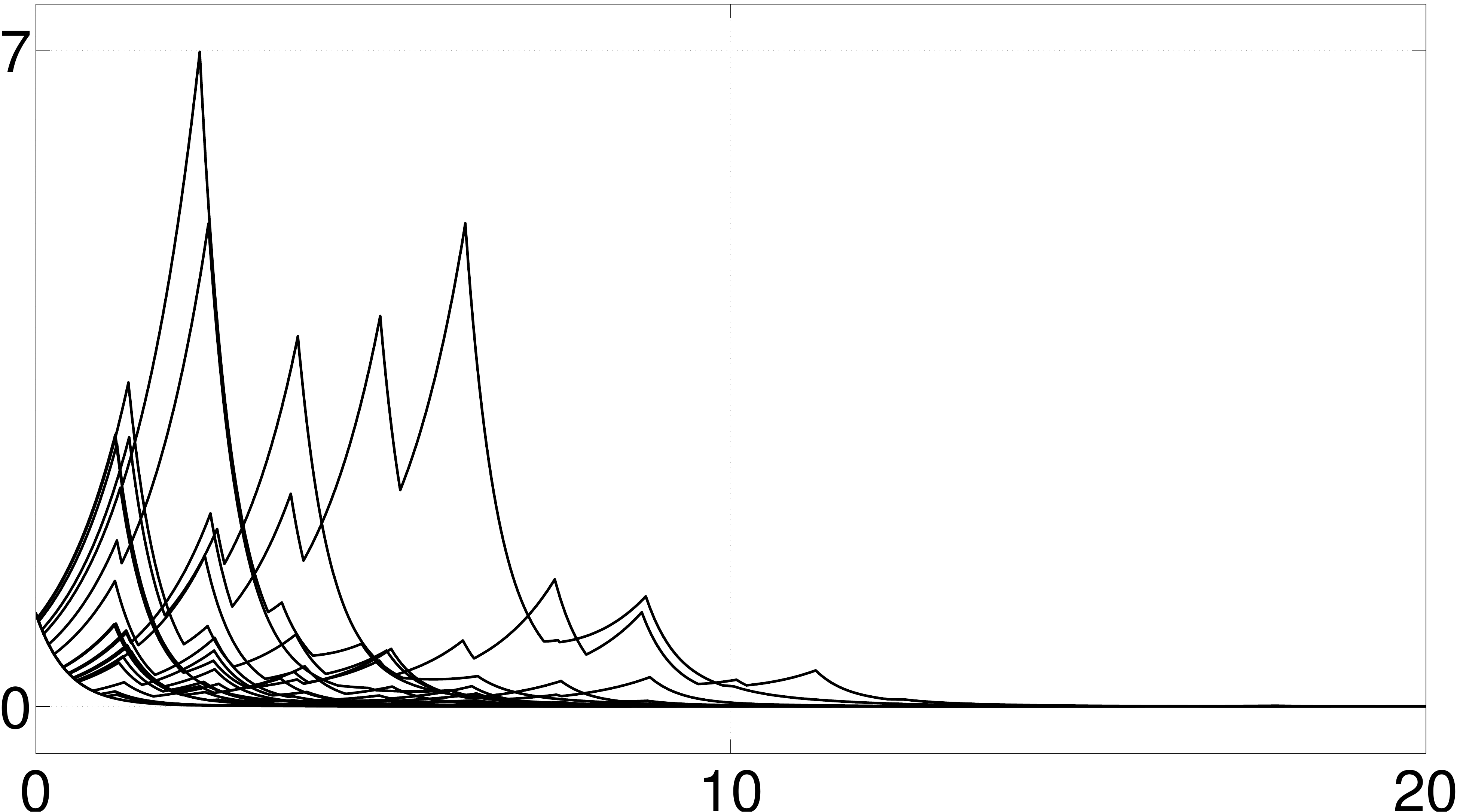}
\caption{20 sample paths of $\norm{x(t)}^2$}
\label{figure:samplepaths}
\end{figure}
\end{example}

\begin{remark}
Theorem~\ref{theorem:main} extends celebrated Floquet's
theory~\cite{Farina2000} for the stability analysis of linear
time-periodic systems. In fact, if the process $\sigma$ is a periodic
function with the period~$T$ then we have $\rho(E[M_0^{[m]}]) =
\rho(M_0^{[m]}) = \rho(M_0)^m$. Therefore, by
Theorem~\ref{theorem:main}, the linear time-periodic system $\Sigma$
is stable in the standard sense if and only if $\rho(M_0) < 1$, which
is the main consequence of Floquet's theory.
\end{remark}

\section{Proof of the main result} \label{sec:proof}

The proof of Theorem~\ref{theorem:main} is based on the discretization
of $\Sigma$ at the embedded renewal process of the underlying
regenerative process. In order to analyze the stability of the
discretization, in the next section we first present the stability
analysis of discrete-time switched linear systems with
i.i.d.~parameters. Then in Section~\ref{subsec:proof} we give the
proof of Theorem~\ref{theorem:main}.

\subsection{Stability of discrete-time linear systems with i.i.d.~parameters} 

Let $\{F_k\}_{k\geq 0}$ be independent and identically distributed
random variables following a distribution~$\mu$ on
$\mathbb{R}^{n\times n}$. Consider the discrete-time switched linear
system
\begin{equation*}
\Sigma_\mu: x_d(k+1) = F_k x_d(k),\ k\geq 0
\end{equation*}
where $x_d(0) = x_0 \in \mathbb{R}^n$ is a constant vector. The mean
stability of $\Sigma_\mu$ is introduced as follows.

\begin{definition}\label{defn:sta:sigmamu} Let $m$ be a positive integer.
\begin{itemize}
\item $\Sigma_\mu$ is said to be \emph{exponentially $m$th mean
stable} if there exist $\alpha>0$ and $\beta > 0$ such that
\begin{equation}\label{eq:mu:expsta}
E[\norm{x_d(k)}^m] \leq \alpha e^{-\beta k} \norm{x_0}^m
\end{equation}
for all $x_0$ \add{and $k\geq 0$}.

\item $\Sigma_\mu$ is said to be \emph{stochastically $m$th mean
stable} if $\sum_{k=0}^\infty E[\norm{x_d(k)}^m] < \infty$ for any
$x_0$.
\end{itemize}
\end{definition}

Also we define the positivity of $\Sigma_\mu$ in the same way as for
$\Sigma$.
\begin{definition}
We say that $\Sigma_\mu$ is positive if $x_0\geq 0$ implies $x(k) \geq
0$ with probability one for every $k \geq 0$.
\end{definition}

When we check the exponential mean stability of a positive
$\Sigma_\mu$, we can without loss of generality assume that its
initial state is nonnegative. 

\begin{lemma}\label{lem:Lemma36result}
Assume that $\Sigma_\mu$ is positive. Then $\Sigma_\mu$ is
exponentially $m$th mean stable if and only if there exist $\alpha>0$
and $\beta > 0$ such that \eqref{eq:mu:expsta} holds for all $x_0 \geq
0$ \add{and $k\geq 0$}.
\end{lemma}
\begin{proof}
See the proof of \cite[Lemma~3.6]{Ogura2013f}.
\end{proof}

The next proposition characterizes the stability of $\Sigma_\mu$ in
terms of the spectral radius of a matrix. 

\begin{proposition}\label{prop:stab:mu:char}
Assume that either
\begin{enumerate}[label={\alph*.},ref={\alph*}]
\item \label{item:mu:even} $m$ is even or

\item \label{item:mu:invariant} $\Sigma_\mu$ is positive.
\end{enumerate}
Then the following conditions are equivalent. 
\begin{enumerate}
\item \label{item:mu:expsta} $\Sigma_\mu$ is exponentially $m$th mean
stable.

\item \label{item:mu:stosta} $\Sigma_\mu$ is stochastically $m$th mean
stable.

\item \label{item:mu:Shcsta} $E[F_0^{[m]}]$ is Schur stable.
\end{enumerate}
\end{proposition}

\begin{proof}
We shall show the cycle [\ref{item:mu:expsta} $\Rightarrow$
\ref{item:mu:stosta} $\Rightarrow$ \ref{item:mu:Shcsta} $\Rightarrow$
\ref{item:mu:expsta}]. One can easily see [\ref{item:mu:expsta}
$\Rightarrow$ \ref{item:mu:stosta}].

[\ref{item:mu:stosta} $\Rightarrow$ \ref{item:mu:Shcsta}]: Let
$x(\cdot;x_0)$ denote the trajectory of $\Sigma_\mu$ with the initial
state~$x_0$. Since the identity \eqref{eq:(AB)^[m]} shows
$E[x(k+1;x_0)^{[m]}] = E[F_0^{[m]}]E[x(k;x_0)^{[m]}]$, an induction
with respect to $k$ yields
\begin{equation}\label{eq:0->k}
E[x(k;x_0)^{[m]}] = E[F_0^{[m]}]^k x_0^{[m]}. 
\end{equation}
Now assume that $\Sigma_\mu$ is stochastically $m$th mean stable. Let
$\lambda$ be an eigenvalue of $E[F_0^{[m]}]$ with a corresponding
eigenvector~$v\in \mathbb{C}^{n_m}$. Since the set $\{x^{[m]}:
x\in\mathbb{R}^n  \}$ spans $\mathbb{R}^{n_m}$
(\cite[Lemma~1.5]{Ogura2013f}), there exist $y_1, \dotsc, y_\ell
\in\mathbb{R}^n$ and $c_1, \dotsc, c_\ell \in \{1, \sqrt{-1}\}$ such
that $v = \sum_{i=1}^{\ell}c_i y_i^{[m]}$. Multiplying
$E[F_0^{[m]}]^k$ to this equation we obtain $\lambda^k v =
\sum_{i=1}^{\ell}c_i E[x(k;y_i)^{[m]}]$ by \eqref{eq:0->k}. Therefore,
by the triangle inequality and \eqref{eq:norm^[m]}, we can see
$\abs{\lambda}^k \norm{v} \leq \sum_{i=1}^{\ell}
E[\norm{x(k;y_i)}^m]$. Since the right hand side of this inequality is
summable with respect to $k$ by the stochastic $m$th mean stability
of $\Sigma_\mu$ and also $\norm{v} \neq 0$, we conclude $\abs \lambda
< 1$.

[\ref{item:mu:Shcsta} $\Rightarrow$ \ref{item:mu:expsta}]: Assume that
$E[F_0^{[m]}]$ is Schur stable. By \eqref{eq:0->k}, there exist
\change{$C>0$}{$\alpha>0$} and $\beta >0$ such that
\begin{equation}\label{eq:pf:base}
\norm{E[x(k)^{[m]}]} 
\leq 
\alpha e^{-\beta k} \norm{x_0^{[m]}} 
= 
\alpha e^{-\beta k} \norm{x_0}^m
\end{equation}
for every $k\geq 0$. We shall show that $\Sigma_\mu$ is exponentially
$m$th mean stable. Let $x_0$ and $k\geq 0$ be arbitrary and write $y
= x(k)$. We consider the two cases \ref{item:mu:even} and
\ref{item:mu:invariant} separately. First assume that $m$ is even.
Take positive constants $C_1$ and $C_2$ such that $C_1 \norm{\cdot}_1
\leq \norm{\cdot} \leq C_2 \norm{\cdot}_m$. Then
\begin{equation}\label{eq:pf:evenm:1}
E[\norm{y}^m] 
\leq 
C_2^m E[\norm{y}^m_m] 
=
C_2^m \sum_{i=1}^{n} \Abs{E[y_i^m]}. 
\end{equation} 
Since the random vector $y^{[m]}$ has all the monomials $y_i^m$
($i=1$, $\cdots$, $m$), we can see that $\sum_{i=1}^{n} \abs{E[y_i^m]}
\leq \norm{E[y]^{[m]}}_1\leq C_1^{-1}\norm{E[y]^{[m]}}$. Therefore
this inequality together with \eqref{eq:pf:evenm:1} and
\eqref{eq:pf:base} shows that $\Sigma_\mu$ is exponentially $m$th mean
stable.

Next assume that $\Sigma_\mu$ is positive. Notice that, by
Lemma~\ref{lem:Lemma36result}, without loss of generality we can
assume $x_0 \geq 0$, which implies $y \geq 0$ and hence $y^{[m]} \geq
0$ with probability one. Let us take a positive constant $C_3$ such
that $\norm{\cdot} \leq C_3 \norm{\cdot}_1$. Then we have $\norm{y}^m
=\norm{y^{[m]}} \leq C_3 \norm{y^{[m]}}_1 = C_3 1_{n_m}^\top y^{[m]}$
with probability one. Therefore, the Schwartz inequality shows
$E[\norm{y}^m] \leq C_3 1_{n_m}^\top E[y^{[m]}] \leq C_3
\norm{1_{n_m}}\, \norm{E[y^{[m]}]}$. This inequality and
\eqref{eq:pf:base} prove the exponential $m$th mean stability of
$\Sigma_\mu$.
\end{proof}

\begin{remark}
Proposition~\ref{prop:stab:mu:char} improves the stability condition
in \cite{Ogura2012b} by reducing the computational cost for checking
mean stability. The size $n_m$ of the matrix $E[F_0^{[m]}]$ is far
less than the size $n^m$ of the matrix used in
\cite[Theorem~5.1]{Ogura2012b}. Also the proof presented above is
simpler than the proof of \cite[Theorem~5.1]{Ogura2012b}, which needs
the approximation of $\mu$ by a sequence of finitely supported
probability measures.
\end{remark}

\subsection{Proof of the main result}\label{subsec:proof}

Let $\Sigma$ be a regenerative switched linear system satisfying the
conditions \ref{item:evenORnon-neg} to \ref{item:Mbdd} and let $x$ be
the trajectory of $\Sigma$.  Then the discretized process
$\{x_d(k)\}_{k\geq 0}$ given by $x_d(k) = x(Z_k)$ is clearly the
solution of the discrete-time system
\begin{equation*}
\mathcal S \Sigma:  x_d(k+1) = M_kx_d(k),\ k\geq 0
\end{equation*}
where $M_k$ is defined by \eqref{eq:def:M_k}.
Proposition~\ref{prop:stab:mu:char} immediately gives the next
corollary on the stability of $\mathcal S \Sigma$.

\begin{corollary}\label{cor:char}
The following conditions are equivalent.
\begin{enumerate}
\item $\mathcal S \Sigma$ is exponentially $m$th mean stable. 

\item $\mathcal S \Sigma$ is stochastically $m$th mean stable. 

\item $E[M_0^{[m]}]$ is Schur stable. 
\end{enumerate}
\end{corollary}

\begin{proof}
The random variables $\{M_k\}_{k=0}^\infty$ are independent and
identically distributed by the definition of regenerating processes.
Also \ref{item:evenORnon-neg} automatically ensures that one of the
conditions \ref{item:mu:even} and \ref{item:mu:invariant} in
Proposition~\ref{prop:stab:mu:char} is satisfied.
\end{proof}

We will also need the next lemma to prove the main result.

\begin{lemma}
There exists $C>1$ such that
\begin{equation}\label{eq:c:C1C2}
C^{-1} \norm{x(Z_k)} \leq \norm{x(t)} \leq C \norm{x(Z_k)}
\end{equation}
for all $k\geq 0$ and $t\in [Z_k, Z_{k+1}]$.
\end{lemma}

\begin{proof}
Let $t \in [Z_k, Z_{k+1}]$. By \ref{item:Mbdd} there exists a
constant, say, $M>0$, such that $\norm{A_s} \leq M$ for every $s\in
\mathcal S$. Since $x(t) = \int_{Z_k}^t A_{\sigma_\tau}x(\tau)\,d\tau
+ x(Z_k)$ we have $\norm{x(t)} \leq \int_{Z_k}^t M
\norm{x(\tau)}\,d\tau + \norm{x(Z_k)}$. Then Gronwall's inequality and
\ref{item:hk<T} shows $\norm{x(t)} \leq
e^{M\mathchangescript{T}{\norm{R_1}}}\norm{x(Z_k)}$\add{, where
$\norm{R_1}$ denotes the essential supremum of $R_1$}. Similarly we
can show $e^{-M\mathchangescript{T}{\norm{R_1}}}\norm{x(Z_k)} \leq
\norm{x(t)}$. This completes the proof.
\end{proof}

Now we prove the main result of this paper.

\begin{proof}[Proof of Theorem~\ref{theorem:main}]
We shall show the cycle [\ref{item:expsta} $\Rightarrow$
\ref{item:stosta} $\Rightarrow$ \ref{item:Schsta} $\Rightarrow$
\ref{item:expsta}]. It is obvious to prove [\ref{item:expsta}
$\Rightarrow$ \ref{item:stosta}].

[\ref{item:stosta} $\Rightarrow$ \ref{item:Schsta}]: By inequality
\eqref{eq:c:C1C2} and the definition of regenerative processes, for
every $k$ we can show
\begin{equation}\label{eq:proof23}
\begin{aligned}
E\left[\int_{Z_k}^{Z_{k+1}} \norm{x(t)}^m\,dt\right]
&\geq
C^{- m} E\left[\norm{x(Z_k)}^m\int_{Z_k}^{Z_{k+1}} dt\right]
\\
&=
C^{- m} E[R_{k+1}] E[\norm{x_d(k)}^m]
\\
&=
C^{- m} E[R_{1}] E[\norm{x_d(k)}^m], 
\end{aligned}
\end{equation}
where we used the definition of regenerative processes. Since Fubini's
theorem shows $\int_0^\infty E[\norm{x(t)}^m]\, dt = \sum_{k=0}^\infty
E\Bigl[ \int_{Z_k}^{Z_{k+1}} \norm{x(t)}^m\,dt \Bigr]$, taking the
summation about $k$ in \eqref{eq:proof23} yields
\begin{equation*}
C^{- m} E[R_1]
\sum_{k=0}^\infty E[\norm{x_d(k)}^m] \leq \int_0^\infty
E[\norm{x(t)}^m]\,dt < \infty.
\end{equation*}
Therefore $\mathcal S\Sigma$ is stochastically $m$th mean stable
because both $C$ and $E[R_1]$ are positive. Hence
Corollary~\ref{cor:char} implies that $E[M_0^{[m]}]$ is Schur stable.

[\ref{item:Schsta} $\Rightarrow$ \ref{item:expsta}]: Here we employ
the idea used in the proof of the sufficiency part for
\cite[Theorem~2.5]{Ogura2013f}. Assume that $E[M_0^{[m]}]$ is Schur
stable. Then $\mathcal S \Sigma$ is exponentially $m$th mean stable
by Corollary~\ref{cor:char}. Let $x_0$ and $t\geq 0$ be arbitrary. Let
us define
$k_t = \max\{ k \in \mathbb{N}: Z_k \leq t \}$.
Since $Z_{k_t} \leq t<Z_{k_t+1}$, the inequality \eqref{eq:c:C1C2}
gives $\norm{x(t)} \leq C \norm{x(Z_{k_t})}= C \norm{x_d(k_t)}$.
Therefore
\begin{equation}\label{eq;pf;Pre}
E[\norm{x(t)}^m] \leq C^m E[\norm{x_d(k_t)}^m].
\end{equation}
On the other hand, since \ref{item:hk<T} shows $t<Z_{k_t + 1} \leq
\mathchange{T}{\norm{R_1}}(k_t +1)$ we have $k_t >
\mathchange{T}{\norm{R_1}}^{-1}t -1$. This implies $\norm{x_d(k_t)}^m
\leq \sum_{k > \mathchangescript{T}{\norm{R_1}}^{-1}t -1}
\norm{x_d(k)}^m$. Taking the expectation in this inequality and then
using the $m$th mean stability of $\mathcal S \Sigma$ we obtain
$E[\norm{x_d(k_t)}^m]  \leq \alpha' e^{-\beta' t} \norm{x_0}^m$, where
$\alpha' = {\alpha e^\beta}/{(1-e^{-\beta})}$ and $\beta' =
\beta/\mathchange{T}{\norm{R_1}}$ \add{are positive constants}. This
inequality and \eqref{eq;pf;Pre} show the $m$th mean exponential
stability of $\Sigma$.
\end{proof}

%
%
%
%
%
%
%

\section{Discrete-time case} \label{sec:disc}

This section briefly discusses the stability characterization of
regenerative switched linear systems in discrete-time. Let $\sigma =
\{\sigma_k\}_{k=0}^\infty$ be a regenerative process taking values in
a set $\mathcal S$ and defined on the set of nonnegative
integers~$\{0, 1, \dotsc \}$. Let $\{A_s\}_{s\in\mathcal S}$ be a
family of ${n\times n}$ real matrices.  Consider the discrete-time
regenerative switched linear system
\begin{equation*}
\Sigma_d: x(k+1) = A_{\sigma_k} x(k),\ k\geq 0.
\end{equation*}
The exponential and stochastic mean stability of $\Sigma_d$ are
defined as in Definition~\ref{defn:sta:sigmamu}. In addition to the
assumptions~\ref{item:evenORnon-neg} to \ref{item:Mbdd} we place the
next assumption on $\Sigma_d$:
\begin{enumerate}[label={A\arabic*. },ref={A\arabic*}]
\setcounter{enumi}{3} \item \label{item:d:invMbdd} $A_s$ is invertible
for each $s\in \mathcal S$ and the set $\{A_s^{-1}\}_{s\in\mathcal S}$
is bounded.
\end{enumerate}
For each $k\geq 0$ we define the transition matrix $M_{d,k}$ for
$\Sigma_d$ representing the transition of $x$ from $t=Z_k$ to
$t=Z_{k+1}$ in the same way as we defined $M_k$ for continuous-time
regenerative switched linear systems in \eqref{eq:def:M_k}. The next
theorem is a discrete-time counterpart of Theorem~\ref{theorem:main}.

\begin{theorem}
The following statements are equivalent.
\begin{enumerate}
\item $\Sigma_d$ is exponentially $m$th mean stable.

\item $\Sigma_d$ is stochastically $m$th mean stable.

\item $E\bigl[M_{d,0}^{[m]}\bigr]$ is Schur stable.
\end{enumerate}
\end{theorem}

\begin{proof}
Let  $\{Z_k\}_{k=0}^\infty$ be the embedded renewal process of
$\sigma$. Using \ref{item:Mbdd} and \ref{item:d:invMbdd} we can show
the existence of a constant $C > 1$ such that, for every $k\geq 0$, if
$Z_k \leq \ell \leq Z_{k+1}$ then $C^{-1} \norm{x_d(Z_k)} \leq
\norm{x_d(\ell)} \leq C \norm{x_d(Z_k)}$. Then we can prove the
desired equivalence in the same way as in the proof of
Theorem~\ref{theorem:main}. The details are omitted.
\end{proof}

\section{Conclusion}

In this paper we investigated the mean stability of regenerative
switched linear systems. A necessary and sufficient condition for the
$m$th mean stability of regenerative switched linear systems
\change{is}{was} established under the assumption that either $m$ is
even or the system is positive and that the length of each cycle of
the underlying regenerative process is essentially bounded. The proof
used a  discretization of the system at the embedded renewal process
of the underlying regenerative process. A numerical example was
presented to illustrate the result.

%











\end{document}